\documentclass[10pt]{amsart}
\usepackage{amsthm}
\usepackage{amsfonts}
\usepackage{amsmath}
\usepackage{amssymb}
\usepackage{mathrsfs}
\usepackage{epsfig}
\usepackage{ifthen}

\newcommand{\ncom}{\newcommand}
\ncom{\ul}{\underline}
\ncom{\ol}{\overline}
\ncom{\bq}{\begin{equation}}
\ncom{\eq}{\end{equation}}
\ncom{\beqn}{\begin{eqnarray*}}
\ncom{\eeqn}{\end{eqnarray*}}
\ncom{\beq}{\begin{eqnarray}}
\ncom{\eeq}{\end{eqnarray}}
\ncom{\nno}{\nonumber}
\ncom{\rar}{\rightarrow}
\ncom{\Rar}{\Rightarrow}
\ncom{\noin}{\noindent}
\ncom{\bc}{\begin{centre}}
\ncom{\ec}{\end{centre}}
\ncom{\sz}{\scriptsize}
\ncom{\rf}{\ref}
\ncom{\sgm}{\sigma}
\ncom{\Sgm}{\Sigma}
\ncom{\dt}{\delta}
\ncom{\Dt}{Delta}
\ncom{\lmd}{\lambda}
\ncom{\Lmd}{\Lambda}
\ncom{\eps}{\epsilon}
\ncom{\pcc}{\stackrel{P}{>}}
\ncom{\dist}{{\textrm\,dist}}
\ncom{\sspan}{{\textrm\,span}}
\ncom{\re}{{\textrm Re\,}}
\ncom{\im}{{\textrm Im\,}}
\ncom{\sgn}{{\textrm sgn\,}}
\ncom{\ba}{\begin{array}}
\ncom{\ea}{\end{array}}
\ncom{\eop}{\hfill{{\rule{2.5mm}{2.5mm}}}}
\ncom{\eoe}{\hfill{{\rule{1.5mm}{1.5mm}}}}
\ncom{\eof}{\hfill{{\rule{1.5mm}{1.5mm}}}}
\ncom{\hone}{\mbox{\hspace{1em}}}
\ncom{\htwo}{\mbox{\hspace{2em}}}
\ncom{\hthree}{\mbox{\hspace{3em}}}
\ncom{\hfour}{\mbox{\hspace{4em}}}
\ncom{\hsev}{\mbox{\hspace{7em}}}
\ncom{\vone}{\vskip 2ex}
\ncom{\vtwo}{\vskip 4ex}
\ncom{\vonee}{\vskip 1.5ex}
\ncom{\vthree}{\vskip 6ex}
\ncom{\vfour}{\vspace*{8ex}}
\ncom{\norm}{\|\;\;\|}
\ncom{\integ}[4]{\int_{#1}^{#2}\,{#3}\,d{#4}}
\ncom{\inp}[2]{\langle{#1},\,{#2} \rangle}
\ncom{\Inp}[2]{\Langle{#1},\,{#2} \Langle}
\ncom{\vspan}[1]{{{\textrm\,span}\#1 \}}}
\ncom{\dm}[1]{\displaystyle {#1}}

\newtheorem{theorem}{\bf Theorem}[section]
\newtheorem{corollary}[theorem]{\bf Corollary}
\newtheorem{lemma}[theorem]{\bf Lemma}

\newtheoremstyle
	{remarkstyle}
	{}
	{11pt}
	{}
	{}
	{\bfseries}
	{:}
	{     }
	{\thmname{#1} \thmnumber{#2} }

\theoremstyle{remarkstyle}

\newtheorem{remark}[theorem]{\bf Remark}
\newtheorem{definition}[theorem]{\bf Definition}

\def \Z{\mathbb{Z}}
\def \N{\mathbb{N}}

\title{On the Diophantine equations $f(x)=g(y)$}
\author{S. Subburam and J. Tanti}
\address{Sastra University, Kumbhakonam, Tamilnadu, India,}
\address{Central University of Jharkhand, Ranchi, Jharkhand, India.}
\email[S. Subburam] {ssrammaths@yahoo.com,}
\email[Jagmohan Tanti] {jagmohan.t@gmail.com.}
\date{}
\begin{document}

\begin{abstract}
The study of finiteness or infiniteness of integer solutions of a Diophantine equation has been considered as a standard problem in the literature. In this paper, for $f(x)\in\Z[x]$ monic and $\displaystyle q_1, \cdots, q_m\in\Z$, we study the conditions for 
which the Diophantine equation $$(y+q_1)(y+q_2)\cdots(y+q_m)=f(x)$$ has finitely many solutions in integers. Also assuming $ABC$-Conjecture, we study the conditions for 
finiteness of integer solutions of the Diophantine equation $f(x)=g(y)$. 
\end{abstract}
\maketitle
\renewcommand{\thefootnote}{}
\footnote{ \noindent\textbf{} \vskip0pt
\textbf{2010 Mathematics Subject Classification:}
Primary 11D41, 11D45, Secondary 11D25.
}
\section{Introduction} Consider the equation 
\begin{equation}
 f(x)=g(y),
\label{e1}
\end{equation} where $f(x)$ and $g(x)$ are polynomials  with integral coefficients of degrees $n$ and $m$ respectively. Already we know the following 
natural problem.
\vskip.5mm
\noin{\bf Problem 0}
Does the equation (\ref{e1}) have finitely or infinitely many solutions in integers?
\vskip.5mm

Several mathematicians solved some special cases of this problem. Yuri F. Bilu and Robert F. Tichy in \cite{Bilu} have given a criteria to check whether 
(\ref{e1}) has 
finitely many solutions in integers with a bounded denominator, when $f(x)$ and $g(x)$ are monic. Baker \cite{Ba} gives an upper bound for integral solutions of 
the equation 
\begin{equation}
y^2=a_nx^n+a_{n-1}x^{n-1}+\cdots+a_0,
\label{hypereq}
\end{equation}
where $\displaystyle a_n , a_{n-1}, \cdots, a_0$ are rational integers, $a_n \neq 0$, when $n \geq 5$ and the polynomial on the right separable.
In a major breakthrough, in early seventies, Baker gives an upper bound for integral solutions of the super elliptic equations $y^m = f(x)$. 
Also \cite{Bi}, \cite{Br}, \cite{Bu}, \cite{Po}, \cite{Sc}, \cite{Sp} and 
\cite{Vo} give some upper bounds for some equations of the form $y^m = f(x)$.  Also \cite{Le} tells that (\ref{hypereq}) has finitely many solutions in 
integers, 
if $a_nx^n + a_{n-1}x^{n-1} +\cdots+ a_0$  has distinct zeros and $n \geq 3$.
\vskip.3mm
Our aim in this paper is to study the conditions for which the equation (\ref{e1}) has finitely many solutions in integers. Here  
without assuming $ABC$ Conjecture, we solve the problem 0 for some special type of $f(x)$ and $g(y)$. When we assume $ABC$ conjecture, we cover much more general 
types of diophantine equations of the form (\ref{e1}) including non-monic polynomials $f(x)$ and $g(y)$. The following three theorems are without assuming 
$ABC$-Conjecture. 
\begin{theorem}
 Let $f(x)\in\Z[x]$ be monic, with $\deg f(x)=n>m$ and $\displaystyle q_1, \cdots, q_m\in\Z$.
Then the equation $$(y+q_1)(y+q_2)\cdots (y+q_m)=f(x)$$ has finitely many solutions in integers if and only 
if $|x-y|$ is bounded for all integer solutions $(x,y)$.
\label{mainthmngtm}
\end{theorem}

\begin{theorem}
Let $f(x)\in\Z[x]$ be monic, $\deg f(x)=m$, at least one of the roots of $f(x)$ be non integral and let $\displaystyle q_1, q_2, \cdots, q_m\in\Z$. Then the 
equation 
$$(y+q_1)(y+q_2)\cdots (y+q_m)=f(x)$$ 
has finitely many solutions $x$ and $y$ in integers.
\label{mainthmneqm}
\end{theorem}

\begin{theorem}
 Let $f(x)\in\Z[x]$ be monic with $\deg f(x)=n<m$ and $\displaystyle q_1,\cdots, q_m\in\Z$. Then the equation 
$$(y+q_1)(y+q_2)\cdots (y+q_m)=f(x)$$ has 
finitely many solutions in integers if and only if $|x-y|$ is bounded for all integer solutios $(x,y)$.
\label{mainthmnltm} 
\end{theorem}
The following three theorems are assuming $ABC$-Conjecture. 
\begin{theorem}
Let $\displaystyle f_{11}(x), \cdots, f_{1l}(x), f_{21}(x), \cdots, f_{2m}(x), g_1(x), \cdots, g_n(x)\in\Z[x]$, $\displaystyle\alpha_1,\alpha_2,\cdots, \alpha_l, \beta_1,\beta_2,\cdots, \beta_m, \gamma_1,\gamma_2, \cdots, \gamma_n\in\Z^+$. Let 
$$f(x)=f_{11}(x)^{\alpha_1}f_{12}(x)^{\alpha_2}\cdots f_{1l}(x)^{\alpha_l}\pm f_{21}(x)^{\beta_1}f_{22}(x)^{\beta_2}\cdots f_{2m}(x)^{\beta_m}$$
and $$g(y)=g_1(y)^{\gamma_1}g_2(y)^{\gamma_2}\cdots g_n(y)^{\gamma_n},$$
with 
\begin{enumerate}
 \item $\deg g(y)<\deg f(x)$
\item $\deg g(y)>1 + \deg (f_{11}(x)f_{12}(x) \cdots f_{1l}(x) f_{21}(x) \cdots f_{2m}(x) g_1(x) \cdots g_n(x))$
\item $f_{11}(x)f_{12}(x) \cdots f_{1l}(x)$ and $f_{21}(x) \cdots f_{2m}(x)$ have no common factor in $\Z[x]$
\item leading coefficient of each of $f_{11}(x) \cdots f_{1l}(x)$, $f_{21}(x) \cdots f_{2m}(x)$,\\ $g_1(x) \cdots g_n(x)$ is positive.
\end{enumerate}
If we assume $ABC$-Conjecture, then the equation 
\begin{equation*}
f(x)=g(y)
\end{equation*}
has finitely many solutions $(x,y)$ in integers.
\label{abcthm1}
\end{theorem}
\begin{theorem}
Let $\displaystyle f_{11}(x), \cdots, f_{1l}(x), f_{21}(x), \cdots, f_{2m}(x), g_1(x), \cdots, g_n(x)\in\Z[x]$, $\displaystyle\alpha_1,\alpha_2,\cdots, \alpha_l, \beta_1,\beta_2,\cdots, \beta_m, \gamma_1,\gamma_2, \cdots, \gamma_n\in\Z^+$. Let 
$$f(x)=f_{11}(x)^{\alpha_1}f_{12}(x)^{\alpha_2}\cdots f_{1l}(x)^{\alpha_l}\pm f_{21}(x)^{\beta_1}f_{22}(x)^{\beta_2}\cdots f_{2m}(x)^{\beta_m}$$
and $$g(y)=g_1(y)^{\gamma_1}g_2(y)^{\gamma_2}\cdots g_n(y)^{\gamma_n},$$
with 
\begin{enumerate}
 \item $\deg g(y)>\deg f(x)$
\item $\deg f(x)>1 + \deg (f_{11}(x)f_{12}(x) \cdots f_{1l}(x) f_{21}(x) \cdots f_{2m}(x) g_1(x) \cdots g_n(x))$
\item $f_{11}(x)f_{12}(x) \cdots f_{1l}(x)$ and $f_{21}(x) \cdots f_{2m}(x)$ have no common factor in $\Z[x]$
\item leading coefficient of each of $f_{11}(x) \cdots f_{1l}(x)$, $f_{21}(x) \cdots f_{2m}(x)$,\\ $g_1(x) \cdots g_n(x)$ is positive.
\end{enumerate}
If we assume $ABC$-Conjecture, then the equation 
\begin{equation*}
f(x)=g(y)
\end{equation*}
has finitely many solutions $(x,y)$ in integers.
\label{abcthm2}
\end{theorem}
\begin{theorem}
 Let $f(x), g(x) \in Z[x]$, $f(x)$ be separable with $\deg f(x)>2$ and $$g(y)=g_1(y)^{\gamma_1}g_2(y)^{\gamma_2}\cdots g_n(y)^{\gamma_n},$$ where 
$\gamma_1$, $\gamma_2$, ..., $\gamma_n$ are positive integers $\geq$ 2, $g_{i}(x) \in Z[x]$ for $ i = 1, 2, ..., k$. If assume $ABC$ Conjecture, 
then the equation $$ f(x) = g(y)$$ has only finitely many solutions in integers.
\label{abcthm3} 
\end{theorem}
As an immediate consequence we have following result.
\begin{corollary}
Assume $ABC$-Conjecture. Let $f(x)\in\Z[x]$ be nonconstant and $\displaystyle q_1, \cdots, q_m\in\Z$.
\begin{enumerate}
\item If $f(x)$ is separable with $\deg f(x)>2$, and for each $j$, $1\leq j\leq m$, $(y+q_i)^2\mid \Pi_{i=1}^m (y+q_i)$, then the equation 
$$\displaystyle(y+q_1)\cdots(y+q_m)=f(x)$$ 
has finitely many integer solutions.
\item
If $m>2$, $\displaystyle q_1, \cdots, q_m$ are distinct and $f(x)=\Pi_{i=1}^kf_i^{l_i}(x)$, where $f_i(x)\in\Z[x]$ and 
$l_i\geq2$ for $1\leq i\leq k$, then the equation $$\displaystyle(y+q_1)\cdots(y+q_m)=f(x)$$ has finitely many integer solutions.
\end{enumerate}
\end{corollary}
Our paper contains five sections. In the second section we state $ABC$-Conjecture. Here We also establish some results to equation (\ref{e1}). 
In the third Section we study the conditions in Cases  $m<n$ and $m>n$, for which the equation (\ref{e1}) has finitely many solutions in integers.
In fourth section, we study Case $m=n$. 
In the fifth section, we study the conditions assuming $ABC$-Conjecture for which the equation (\ref{e1}) has finitely many solutions in integers.
\section{Preliminaries} 
\vskip3mm
\noin {\bf The $ABC$-Conjecture.} \textit{(Oesterl$\acute{e}$, Masser, Szpiro)}: Fix $\epsilon>0$. If $\displaystyle a, b, c$ are coprime positive integers 
satisfying $a+b=c$ then $$c\ll_{\epsilon} N(a,b,c)^{1+\epsilon},$$ where $N(a,b,c)$ is the product of distinct prime divisors of $abc$.
\vskip.5mm
The following result is a consequence of $truth$ of $ABC$-conjecture and is a part of Theorem 6 in \cite{Gr}.
\begin{lemma}
Assume that the $ABC$-Conjecture is true. Fix $\epsilon>0$. If $g(x)\in\Z[x]$ has no repeated roots and $q^2\mid g(m)$, then $q\ll_\epsilon |m|^{1+\epsilon}$.
\label{abclemma}
\end{lemma}
\vskip.2mm
An integer is called powerful if $p^2$ divides $n$ for every prime $p$ dividing $n$.
\vskip.2mm
\begin{remark}
Lemma \ref{abclemma} implies that if $g(x)\in\Z[x]$ has degree $>2$ and is separable, then $g(m)$ is powerful for only finitely many integers $m$.
\label{remark}
\end{remark}
The following results are without assuming $ABC$-Conjecture.
\begin{lemma}
Let $\displaystyle f(x), g(x)\in\Z[x]$ with $\displaystyle \deg f(x)>\deg g(x)$.
Then the equation $$f(x)=g(y)$$ has finitely many integer solutions if and only if $|x-y|$ is bounded for all 
integral
solutions $(x,y)$.
\label{mneqn}
\end{lemma}
\begin{proof}
Let $f(x)=a_0+\cdots+a_nx^n$ and $g(x)=b_0+\cdots+b_mx^m$ with $a_n\neq0\neq b_m$. Given $n>m$.
If possible assume that the equation $g(y)=f(x)$ has infinitely many integer solutions $x=a$, $y=b$, with $|a-b|$ bounded, i.e. $a-b$ takes values in a finite 
subset $S$ (say) of $\Z$. There should be an integer $c\in S$, such that the equation $g(b)=f(a)$ is satisfied by infinitely many 
$\displaystyle a, b\in\Z, b-a=c$. As $b=a+c$, we have $g(a+c)=f(a)$ for infinitely many integers $a$. 
i.e., $f(a)=g(c)+\frac{g^1(c)}{1!}a+\frac{g^2(c)}{2!}a^2+\cdots+\frac{g^m(c)}{m!}a^m$ for infinitely many integers $a$. 
This implies that the polynomial 
$$a_0-g(c)+\left(a_1-\frac{g^1(c)}{1!}\right)x+\cdots
+\left(a_m-\frac{g^m(c)}{m!}\right)x^m+a_{m+1}x^{m+1}+\cdots+a_nx^n$$ 
over $\Z$ has infinitely many integers roots $x=a$. Therefore
\vskip.2mm
\noin $a_0-g(c)=a_1-g^1(c)=\cdots=a_m-\frac{g^m(c)}{m!}=a_{m+1}=\cdots=a_n=0$, which is a contradiction to the hypothesis.
\end{proof}
For $c$ an integer Let us define a matrix $A(c)$ as follows:
\vskip.5mm
\[A(c)=\left(\begin{array}{ccccccc}
1&c&c^2&\cdots&\cdots&c^{m-1}&c^m\\
0&1&{2\choose1}c&\cdots&\cdots&{m-1\choose m-2}c^{m-2}&{m\choose m-1}c^{m-1}\\
0&0&1&\cdots&\cdots&{m-1\choose m-3}c^{m-3}&{m\choose m-2}c^{m-2}\\
\vdots&\vdots&\vdots&\cdots&\cdots&\vdots&\vdots\\
0&0&0&\cdots&\cdots&1&{m\choose1}c\\
0&0&0&\cdots &\cdots&0&1             
\end{array}\right).\]
\vskip.3mm
\begin{lemma}
 Let $\displaystyle a_0+\cdots+a_mx^m=f(x), b_0+\cdots+b_mx^m=g(x)\in\Z[x]$ with $\deg f(x)=\deg g(x)=m\geq1$. Then the equation $$g(y)=f(x)$$ has finitely many 
solutions $(x,y)$ in integers if and only if $|x-y|$ is bounded and $\displaystyle A(y-x)(b_0,b_1,\cdots,b_m)^t\neq(a_0,a_1,\cdots,a_m)^t$.
\end{lemma}
\begin{proof}
If possible let the equation $g(y)=f(x)$ has infinitely many integer solutions $x=a$, $y=b$, $|a-b|$ bounded and 
for $c=b-a$,\\ $\displaystyle A(c)(b_0,b_1,\cdots,b_m)^t\neq(a_0,a_1,\cdots,a_m)^t$. This implies that there should be at least one integer $c$ such that 
$b-a=c$ and $g(b)=f(a)$ are satisfied by infinitely many integers $a,\,b$. 
Thus we have $f(a)=g(a+c)=g(c)+\frac{g^1(c)}{1!}a+\frac{g^2(c)}{2!}a^2+\cdots+\frac{g^m(c)}{m!}a^m$ for infinitely many integers $a$. This implies 
\begin{eqnarray*}
&&b_0+b_1c+b_2c^2+\cdots+b_mc^m=g(c)=a_0,\\
&&b_1+{2\choose1}b_2c+{3\choose2}b_3c^2+\cdots+{m\choose m-1}b_mc^{m-1}=\frac{g^1(c)}{1!}=a_1,\\
&&b_2+{3\choose1}b_3c+{4\choose2}b_4c^2+\cdots +{m\choose m-2}b_mc^{m-2}=\frac{g^2(c)}{2!}=a_2,\\
&&~~~~~~~~~~~~~~~~~~~~~~~~~~~~~~~~~~~~~~~~~\vdots\\
&&b_{m-2}+{m-1\choose1}b_{m-1}c+{m\choose2}b_mc^2=\frac{g^{m-2}(c)}{(m-2)!}=a_{m-2},\\
&&b_{m-1}+{m\choose1}b_mc=\frac{g^{m-1}(c)}{(m-1)!}=a_{m-1},\\
&&b_m=\frac{g^m(c)}{m!}=a_m.
\end{eqnarray*}
Which is same as $\displaystyle A(c)(b_0,b_1,\cdots,b_m)^t=(a_0,a_1,\cdots,a_m)^t$ a contradiction to our assumption. 
\end{proof}
\begin{corollary}
Let $f(x)$, $g(x)\in\Z[x]$ be as in the lemma, with $a_m\neq b_m$. Then the equation $$g(y)=f(x)$$ has finitely many integer solutions if and only if
$|x-y|$ is bounded for all integer solutions $(x,y)$. 
\end{corollary}
\begin{proof}
Immediate from the Lemma.
\end{proof}
\begin{corollary}
Let $f(x)$, $g(x)\in\Z[x]$ be as in the lemma, $\deg f(x)=\deg g(x)=m\geq2$ and $a_{m-1}-b_{m-1}$ not divisible by $mb_m$. Then the equation $$g(y)=f(x)$$ has 
finitely many integer solutions if and only if $|x-y|$ is bounded for all integer solutions $(x,y)$.
\end{corollary}
\begin{proof}
Let us assume that the equation $g(y)=f(x)$ is satisfied by infinitely many integers $x=a$, $y=b$ with $|a-b|$ bounded.
 So by the lemma we have $b_{m-1}+mb_mc=a_{m-1}$, hence $mb_m$ divides $b_{m-1}-a_{m-1}$, which is a contradiction. 
\end{proof}
\section{Cases $m<n$ and $m>n$}
\noin{\bf Proof of theorem \ref{mainthmngtm}}: As $\displaystyle\deg f(x)=n>m=\deg ((y+q_1)\cdots(y+q_m))$ by Lemma \ref{mneqn} the theorem immediately follows.$\hfill\square$
\vskip.5mm
\noin{\bf Proof of theorem \ref{mainthmnltm}}: As $\displaystyle\deg ((y+q_1)\cdots(y+q_m))=m>n=\deg f(x)$, immediate from Lemma \ref{mneqn}.$\hfill\square$
\section{Case $m=n$} In this section, for $f(x)\in \Z[x]$, monic with $\deg f(x)=m$, we study the conditions for which the equation $$(y+q_1)\cdots(y+q_m)=f(x)$$ 
has finitely many solutions.
\subsection{THEOREMS ON MONIC POLYNOMIALS} In this subsection, we give the definition for a complete composite number of a monic polynomial with positive 
integral coefficients. Using this definition we prove one main theorem on polynomials. This theorem is needed to prove the all theorems of the next subsection.
\begin{definition} 
Let $f(x)$ be a polynomial of degree $m$ with positive integral coefficients. Then a positive integer $k$ is said to be complete composite of $f(x)$, 
if $k = f(n)$, for some positive  integer $n$  and 
$$k  = \Pi_{i=1}^m( n  +  a_i ),$$
where  $\displaystyle a_1, \cdots, a_m$  are  positive  integers. If $f(a)$ is complete composite of $f(x)$, then $a$ is called a complete  number of $f(x)$.
\end{definition}
\begin{lemma}
Let $f(x)$ be a monic polynomial of degree $n$  with positive integral coefficients. If there are infinitely many complete composites of $f(x)$, then $f(x)$ has 
a factor $x + c$ with positive integer $c$.
\label{compositelemma}
\end{lemma}
\begin{proof}
Let $f(x)$ be a monic polynomial of degree $n$ with positive integral coefficients. Let $$f(x)  =  x^n  + a_1 x^{n-1} +\cdots + a_n,$$ where 
$\displaystyle a_1, a_2,\cdots, a_n$ are positive integers. Let there be infinitely many complete composites of $f(x)$. 
Therefore there are infinitely many positive integers $m$ such that 
\begin{equation}
f(m) =  \Pi_{ i = 1}^n(m + b_{im}),
\label{compositeq}
\end{equation}
where $\displaystyle b_{1m}, b_{2m}, \cdots, b_{nm}$ are positive integers. If $b_{im} = c$, for infinitely many pairs $(i, m)$, where $c$ is a positive integer, 
then $m + c$ divides $f(m)$, for infinitely many positive integers $m$. For $f(x)$ and $x + c$, there exist two polynomials $g(x)$ and $h(x)$ with integral 
coefficients such that
$$f(x)  = (x + c)g(x) + h(x),$$
where $h(x)$ is a constant polynomial (see \cite{He}). Since  $\displaystyle f(n), n + c$ and $g(n)$ are integers for any integer $n$, we have
$$f(n)   = ( n + c)g(n) + h(n),$$
where $h(n)$ is an integer. Since $n + c$ divides $f(n)$, for infinitely many positive integers $n$, $n + c$ divides $h(n)$, for infinitely many positive 
integers $n$. Therefore $h(n) = 0$, because $h(x)$ is a constant polynomial. So  
$$f(x) = (x + c) g(x).$$
This gives that $f(x)$ has a factor $x + c$ with positive integer $c$.
\vskip.2mm 
Suppose that $b_{im}$ is not a constant for infinitely many pairs $(i, m)$. From this, it is clear that for each positive integer $k$, there exist a positive 
integer $N$ such that
$$b_{im}  > k,$$
for all $\displaystyle i = 1, 2, \cdots, n$ and for all $m\geq N$, where $m$ is a complete number of $f(x)$. From equation (\ref{compositeq}), we get
\begin{equation}
\sum_{i = 1}^na_im^{n-i}   =  \sum_{i = 1}^nb_{im} m^{n-1}+\sum_{i<j}b_{im}b_{jm}m^{n-2}+\cdots+  \Pi_{i=1}^n b_{im},
\label{compositeeq}
\end{equation}
Since for each positive integer $k$, there exist a positive integer $N$ such that
$$b_{im}  > k,$$
for all $\displaystyle i = 1, 2,\cdots, n$ and for all $m\geq N$, where $m$ is a complete number of $f(m)$, then we can choose a complete number $d$ of $f(x)$ 
such that
$$\sum_{i = 1}^nb_{id}> a_1,\,\,\,\,\sum_{i < j}b_{id}b_{jd}>a_2,\, \cdots\,\cdots,\,\,\,\,\Pi_{i=1}^nb_{id}>a_n.$$
Therefore
$$\sum_{i = 1}^na_id^{n-i}< \sum_{i =1}^nb_{id} d^{n-1} + \sum_{i<j}b_{id}b_{jd}d^{n-2}+ \cdots + \Pi_{i = 1}^n b_{id},$$
which is a contradiction to (\ref{compositeeq}). So our assumption is wrong. This proves  the theorem.
\end{proof}
\begin{lemma}
Let $f(x)$ be a monic polynomial of degree $n$  with positive integral coefficients. Then there are infinitely many complete composites of $f(x)$ if and only if 
$f(x)$ is the product of $n$ one degree monic polynomials with positive integral coefficients.
\label{compositethm}
\end{lemma}
\begin{proof}
Let $f(x)$ be a monic polynomial of degree $n$ with positive integral coefficients. Let there be infinitely many complete composites of $f(x)$.
We prove this result using Mathematical induction on $n$. Clearly the result is true for $n = 1$. Let us assume that the result is true for all
positive integers $< n$. We prove this result for $n$. Since there are infinitely many complete composites of $f(x)$, by Lemma \ref{compositelemma}, 
$f(x)$ has a factor $x + c$ with positive integer $c$. Therefore 
$$f(x) = (x + c)g(x),$$
where $g(x)$ is a polynomial of degree $n-1$ with positive integral coefficients. Since there are infinitely many complete composites of $f(x)$, 
there are infinitely many complete composites of $g(x)$. By our induction hypothesis, $g(x)$ is the product of $n-1$ monic polynomials of degree one 
with positive integral coefficients. Therefore $f(x)$ is the product of $n$ monic polynomials of degree one with positive integral coefficients. 
Converse part is obvious.
\end{proof}
\begin{lemma}
Let $f(x)$ be a monic polynomial of degree $n$  with integral coefficients. Then there are infinitely many complete composites of $g(x) =  f(x+ h)$, 
(where $h$ is a positive integer such that the coefficients of $g(x)$ are positive) if and only if $f(x)$ is the product of $n$ monic polynomials of one 
degree with integral coefficients.
\label{completecompthm}
\end{lemma}
\begin{proof}
Let $f(x)$ be a monic polynomial of degree $n$ with integral coefficients. Let there be infinitely many complete composites of $g(x) = f(x+ h)$, 
where $h$ is a positive integer such that $g(x)$ has positive coefficients. By Lemma \ref{compositethm}, $g(x)$ is the product of $n$ monic polynomials 
of degree one with positive integral coefficients. Let 
$$g(x) = \Pi_{i = 1}^n (x +  m_i).$$
Therefore 
$$f(x) = g(x-h ) = \Pi_{i = 1}^n (x +  m_i-h).$$
Therefore $f(x)$ is the product of $n$ one degree monic polynomials with  integer coefficients.  Converse part is obvious.
\end{proof}
\subsection{Finiteness of integral solutions} In this subsection, we study the finiteness of integral solutions of equation (\ref{e1}) using the results 
from above subsection.
\begin{lemma}
Let $f(x)$ be a monic polynomial of degree $m$ with positive integral coefficients, all roots of $f(x)$ be not integers and let 
$\displaystyle q_1, q_2,\cdots, q_m$ be positive integers. Then the equation
\begin{equation}
( y + q_1)( y + q_2 )\cdots( y + q_m ) =  f(x)
\label{meqnposeq}
\end{equation}
has finitely many solutions $x$ and $y$ in positive integers.
\label{meqnposthm}
\end{lemma}
\begin{proof}
Let $f(x)$ be a monic polynomial of degree $m$ with positive integral coefficients, all roots of $f(x)$ be not integers and let 
$\displaystyle q_1, q_2, \cdots, q_m$ be positive integers. Let
$$q_i  = \max\{q_1, q_2, \cdots, q_m\}.$$
Suppose that there is a solution $x = a$ and $y = b$ in positive integers for (\ref{meqnposeq}) such that
$$b + q_i < a.$$
Therefore
$$(b + q_1)(b + q_2)\cdots(b + q_m)<a^m.$$
Since $a^m < f(a)$, $$(b + q_1)(b + q_2 )\cdots(b + q_m ) < f(a).$$
Which is a contradiction, because (\ref{meqnposeq}) has the solution $x = a$ and $y = b$. Therefore
$$b + q_i \geq a.$$
This implies that $a\in\{\displaystyle b,  b + 1,  b + 2, \cdots, b + q_i - 1,  b + q_i\}$ or $b > a$. Suppose that for (\ref{meqnposeq}), 
there are infinitely many solutions $x = a$ and $y = b$ in positive integers such that $a = b + c$, where $c$ is a positive integer such that 
$0 \leq c \leq q_i$. Then $a - c + q_j$  divides $f(a)$, for $\displaystyle j = 1, 2, \cdots, m$ and infinitely many positive integers $a$. 
This implies that all roots of $f(x)$ are integers. Which is a contradiction to our assumption. Therefore there are infinitely many solutions 
$x= a$ and $y = b$ in positive integers for (\ref{meqnposeq}) such that
$$b > a.$$
This implies that there are infinitely many complete composites of $f(x)$. By Lemma (\ref{compositethm}), $f(x)$ has only integral roots. 
Which is a contradiction to our assumption. This proves the lemma.
\end{proof}
\begin{lemma}
Let $f(x)$ be a monic polynomial of degree $m$ with integral coefficients, all the roots of $f(x)$ be not integers and let $\displaystyle q_1, q_2, \cdots, q_m$ 
be  integers. Then the equation (\ref{e1}),
$$( y + q_1)( y + q_2 )\cdots( y + q_m ) =  f(x)$$
has finitely many solutions $x$ and $y$ in positive integers.
\label{meqnpossolthm}
\end{lemma}
\begin{proof}
Let $f(x)$  be a monic  polynomial of degree $m$ with integral coefficients, all roots of $f(x)$ be not integers and let $\displaystyle q_1, q_2, \cdots, q_m$ 
be 
integers.  Let
$$g(x) = f(x + h),$$
where $h$ is a positive integer such that the coefficients of $g(x)$ are positive integers. Let
$$k = \max\{ | q_1|, | q_2 |, \cdots, | q_m | \}.$$
Clearly $k + 1 + q_i$ is a positive integer for all $\displaystyle i = 1, 2, \cdots, m$. 
Let $p_i = k + 1 + qi$. Since all roots of $f(x)$ are not integers, by Lemma \ref{meqnposthm}, the equation 
$$( y + p_1)( y + p_2 )\cdots( y + p_m ) = g(x)$$
has finitely many solutions in positive integers. This proves the theorem.
\end{proof}
\begin{lemma}
Let $f(x)$ be a monic polynomial of degree $m$ with integral coefficients, all the roots of $f(x)$ be not integers and let 
$\displaystyle q_1, q_2, \cdots, q_m$ be  integers. Then the equation (\ref{e1}),
$$( y + q_1)( y + q_2 )\cdots( y + q_m ) =  f(x)$$
has finitely many solutions $x$ and $y$ in negative integers.
\label{meqnnegsolthm}
\end{lemma}
\begin{proof}
Let $f(x)$ be a monic polynomial of degree $m$ with integral coefficients, all the roots of $f(x)$ be not integers and let 
$\displaystyle q_1, q_2, \cdots, q_m$ be integers. Suppose that there are infinitely many solutions $x = -a$ and $y = -b$ for (\ref{e1}), 
where $a$ and $b$ are positive integers. Therefore there are infinitely many solutions $x = a$ and $y = b$ in positive integers for the equation
\begin{equation}
(- 1)^m ( b - q_1)( b - q_2 )\cdots( b - q_m ) =  f(- a).
\label{meqnnegsoleq}
\end{equation}
If $m$ is even, then (\ref{meqnnegsoleq}) implies 
\begin{equation}
( b- q_1)( b- q_2 )\cdots( b - q_m ) =  f(- a),
\label{meqnnegsoleq1}
\end{equation}
where $f(- x)$ is a monic polynomial with integral coefficients. Therefore (\ref{meqnnegsoleq1}) has finitely many solutions $a$ and $b$ 
in positive integers, by Lemma \ref{meqnpossolthm}.This gives a contradiction to that (\ref{meqnnegsoleq}) has infinitely many solutions $a$ and $b$. 
\vskip.3mm
If $m$ is odd, then (\ref{meqnnegsoleq}) implies
\begin{equation}
( b-q_1)( b-q_2 )\cdots( b-q_m ) =  - f(- a),
\label{meqnnegsoleq2}
\end{equation}
where $-f(- x)$ is a monic polynomial with integral coefficients. Therefore the equation (\ref{meqnnegsoleq2}) has finitely many solutions in positive integers, 
by Lemma \ref{meqnpossolthm}. This gives a contradiction to that (\ref{meqnnegsoleq}) has infinitely many solutions $a$ and $b$. Therefore  (\ref{e1}) 
has finitely many solutions $x$ and $y$ in negative integers.
\end{proof}
\vskip.5mm
\noin {\bf Proof of Theorem \ref{mainthmneqm}}:
Let $f(x)$  be a monic  polynomial of degree $m$ with integral coefficients, all roots of $f(x)$ be not integers and let $\displaystyle q_1, q_2, \cdots, q_m$ 
be integers. Suppose that there are infinitely many solutions $x = - a$ and $y = b$ for the equation (\ref{e1}) such that $a$ and $b$ are positive integers. 
Therefore there are infinitely many solutions $x = a$ and $y = b$ in positive integers for the equation
\begin{equation}
( b + q_1)( b + q_2 )\cdots( b + q_m ) =  f(- a),
\label{meqnnegposeq}
\end{equation}
from (\ref{e1}). If $m$  is even, by Lemma \ref{meqnpossolthm}, the equation (\ref{meqnnegposeq}) has finitely many solutions in positive integers. 
Which is a contradiction to our assumption. If $m$ is odd, then $f(- k)$ is negative for all positive integers $k > l$, where $l$ is some positive integer. 
Which is a contradiction to $b + q_i$ is positive for all $\displaystyle i = 1, 2, \cdots, m$ and for infinitely many positive integers $b$. Suppose that 
there are infinitely many solutions  $x = a$ and $y = - b$  for the equation  (\ref{e1}) such that $a$  and $b$ are positive integers. Therefore
\begin{equation}
(- 1)^m( b - q_1)( b - q_2 )\cdots( b - q_m ) =  f( a).
\label{meqnposnegeq}
\end{equation}
If $m$  is even, by Lemma \ref{meqnpossolthm}, there are finitely many solutions in positive integers. Which is a contradiction to our assumption. 
If $m$ is odd, then we get a contradiction, because left hand side of (\ref{meqnposnegeq}) is negative for infinitely many positive integers $b$,
but right hand side of (\ref{meqnposnegeq}) is positive for all positive integers $a$. From the above two cases, we get that the equation (\ref{e1}) 
has infinitely many solutions $x = a$ and $y =  b$ in positive integers or  the equation (\ref{e1}) has infinitely many solutions $x = - a$ and $y = - b$,
where $a$ and $b$ are positive integers. This contradicts the two Lemmas (\ref{meqnpossolthm}) and (\ref{meqnnegsolthm}). This proves the theorem.
\hfill$\square$
\section{Diophantine Applications of $ABC$ Conjecture}
\noin{\bf Proof of Theorem \ref{abcthm1}}: 
Suppose that there are infinitely many solutions $(x,y)\in\Z\times\Z$ for the equation $(\ref{e1})$. First we assume that the equation $(\ref{e1})$ has infinitely 
many solutions $(a,b)$ in positive integers. Then $$f(a)=g(b)$$ for infinitely many $(a,b)\in\N\times\N$. Therefore 
\begin{equation}
f_{11}(a)^{\alpha_1}f_{12}(a)^{\alpha_2}\cdots f_{1l}(a)^{\alpha_l}\pm f_{21}(a)^{\beta_1}f_{22}(a)^{\beta_2}\cdots f_{2m}(a)^{\beta_m}=g(b)
\label{subburameq}
\end{equation}
for infinitely many $(a,b)\in\N\times\N$.
Now we shall prove that for infinitely many $(a,b)\in\N\times\N$, 
\begin{equation}
\gcd ({f_{11}(a) \cdots f_{1l}(a), f_{21}(a) \cdots f_{2m}(a)})=1.
\label{subbugcd1}
\end{equation}
Suppose for infinitely many $(a,b)$, $$\gcd ({f_{11}(a) \cdots f_{1l}(a), f_{21}(a) \cdots f_{2m}(a)})>1,$$
then by division algorithm, we will have a non constant common factor of the polynomials $f_{11}(x) \cdots f_{1l}(x)$ and $f_{21}(x) \cdots f_{2m}(x)$ which is 
a contradiction to the hypothesis (3) of the theorem. This establishes equation (\ref{subbugcd1}).
Since $f_{11}(x) \cdots f_{1l}(x)$, $f_{21}(x) \cdots f_{2m}(x)$ and $g_1(y) \cdots g_n(y)$ have positive leading coefficients, we get positive integers $N_1$ and $N_2$ such that for all integers $x\geq N_1$ and $y\geq N_2$, $f_{11}(x) \cdots f_{1l}(x)$, $f_{21}(x) \cdots f_{2m}(x)$ and $g_1(y) \cdots g_n(y)$ are positive integers. Let $\epsilon>0$ be a real number such that 
$$\epsilon\deg (f_{11}(x) \cdots f_{1l}(x)f_{21}(x) \cdots f_{2m}(x)g_1(x) \cdots g_n(x))<1.$$
\noin {\bf Case i.} Assume that 
$$f_{11}(a)^{\alpha_1}f_{12}(a)^{\alpha_2}\cdots f_{1l}(a)^{\alpha_l}+ f_{21}(a)^{\beta_1}f_{22}(a)^{\beta_2}\cdots f_{2m}(a)^{\beta_m}=g(b)$$ for infinitely
 many $(a, b) \in \N \times \N$ with
\begin{enumerate}
 \item $a\geq N_{1}$ and $b \geq N_{2}$
 \item $f_{11}(a)^{\alpha_1}f_{12}(a)^{\alpha_2}\cdots f_{1l}(a)^{\alpha_l}, f_{21}(a)^{\beta_1}f_{22}(a)^{\beta_2}\cdots f_{2m}(a)^{\beta_m},g(b) \in \N $
\item $f_{11}(a)^{\alpha_1}f_{12}(a)^{\alpha_2}\cdots f_{1l}(a)^{\alpha_l}, f_{21}(a)^{\beta_1}f_{22}(a)^{\beta_2}\cdots f_{2m}(a)^{\beta_m},g(b)$ are coprimes.  
\end{enumerate}
So by $ABC$ Conjecture, for a fixed $\epsilon>0$, there exists a constant $c_{1}$ such that 
$$ g(b) \leq c_{1}rad\left( f_{11}(a)^{\alpha_1}f_{12}(a)^{\alpha_2}\cdots f_{1l}(a)^{\alpha_l}f_{21}(a)^{\beta_1}
\cdots f_{2m}(a)^{\beta_m}g(b)\right) ^{1+ \epsilon } $$ 
$$= c_{1}rad\left( f_{11}(a)^{\alpha_1}\cdots f_{1l}(a)^{\alpha_l}f_{21}(a)^{\beta_1}\cdots f_{2m}(a)^{\beta_m}g_1(b)^{\gamma_1} 
\cdots g_n(b)^{\gamma_n}\right) ^{1+ \epsilon }.$$
So
$$g(b) \leq c_{1}rad\left( f_{11}(a)f_{12}(a)\cdots f_{1l}(a)f_{21}(a)f_{22}(a)\cdots f_{2m}(a)g_1(b)\cdots g_n(b)\right) ^{1+ \epsilon }.$$
Therefore 
$$g(b) \leq c_{1}\left( f_{11}(a)f_{12}(a)\cdots f_{1l}(a)f_{21}(a)f_{22}(a)\cdots f_{2m}(a)g_1(b)\cdots g_n(b)\right) ^{1+ \epsilon }.$$
\noin{\bf Case ii.} Assume that
\begin{equation}
f_{11}(a)^{\alpha_1}f_{12}(a)^{\alpha_2}\cdots f_{1l}(a)^{\alpha_l} - f_{21}(a)^{\beta_1}f_{22}(a)^{\beta_2}\cdots f_{2m}(a)^{\beta_m}=g(b)
\label{equation4}
\end{equation}
for infinitely many $(a, b) \in \N \times \N$ with
\begin{enumerate}
 \item $a\geq N_{1}$ and $b \geq N_{2}$
 \item $f_{11}(a)^{\alpha_1}f_{12}(a)^{\alpha_2}\cdots f_{1l}(a)^{\alpha_l}, f_{21}(a)^{\beta_1}f_{22}(a)^{\beta_2}\cdots f_{2m}(a)^{\beta_m},g(b) \in \N $
\item $f_{11}(a)^{\alpha_1}f_{12}(a)^{\alpha_2}\cdots f_{1l}(a)^{\alpha_l}, f_{21}(a)^{\beta_1}f_{22}(a)^{\beta_2}\cdots f_{2m}(a)^{\beta_m},g(b)$ are coprimes.  
\end{enumerate}
So by $ABC$ Conjecture, for  a fixed $\epsilon>0$, there exists a constant $c_2$ such that 
$$ f_{11}(a)^{\alpha_1}\cdots f_{1l}(a)^{\alpha_l}\leq c_2rad\left( f_{11}(a)^{\alpha_1}
\cdots f_{1l}(a)^{\alpha_l}f_{21}(a)^{\beta_1}\cdots f_{2m}(a)^{\beta_m}g(b)\right) ^{1+ \epsilon } $$ 
$$= c_2rad\left( f_{11}(a)^{\alpha_1}\cdots f_{1l}(a)^{\alpha_l}f_{21}(a)^{\beta_1}
\cdots f_{2m}(a)^{\beta_m}g_1(b)^{\gamma_1} \cdots g_n(b)^{\gamma_n}\right) ^{1+ \epsilon }.$$
Also as by equation (\ref{equation4}), $g(b)\leq f_{11}(a)^{\alpha_1}\cdots f_{1l}(a)^{\alpha_l}$, we have
$$ g(b) \leq c_{2}rad\left( f_{11}(a)^{\alpha_1}\cdots f_{1l}(a)^{\alpha_l}f_{21}(a)^{\beta_1}
\cdots f_{2m}(a)^{\beta_m}g(b)\right) ^{1+ \epsilon } $$ 
$$= c_2rad\left( f_{11}(a)^{\alpha_1}\cdots f_{1l}(a)^{\alpha_l}f_{21}(a)^{\beta_1}
\cdots f_{2m}(a)^{\beta_m}g_1(b)^{\gamma_1} \cdots g_n(b)^{\gamma_n}\right) ^{1+ \epsilon }.$$
So
$$g(b) \leq c_{2}rad\left( f_{11}(a)f_{12}(a)\cdots f_{1l}(a)f_{21}(a)f_{22}(a)\cdots f_{2m}(a)g_1(b)\cdots g_n(b)\right) ^{1+ \epsilon }.$$
Therefore 
$$g(b) \leq c_{2}\left( f_{11}(a)f_{12}(a)\cdots f_{1l}(a)f_{21}(a)f_{22}(a)\cdots f_{2m}(a)g_1(b)\cdots g_n(b)\right) ^{1+ \epsilon }.$$
So from Case i and Case ii, we get 
\begin{equation*}
 g(b) \leq c\left( f_{11}(a)f_{12}(a)\cdots f_{1l}(a)f_{21}(a)f_{22}(a)\cdots f_{2m}(a)g_1(b)\cdots g_n(b)\right) ^{1+ \epsilon },
\label{equation 5}
\end{equation*}
where $c = max \left \lbrace c_{1}, c_{2}\right\rbrace $.
If $b < a$ for infinitely many $(a, b) \in \N \times \N $, then $f(b) < f(a)$, since leading coefficients of $f(x)$ is positive. Therefore
$$f(b) < f(a) = g(b),$$
so
$$f(b) < g(b).$$
This implies that $\deg g(x) > \deg f(x)$. This contradicts to $(i)$. So $b \geq a $ for infinitely many $(a, b) \in \N \times \N$. Therefore from (5), we can 
write $$ g(b) \leq c\left( f_{11}(b)f_{12}(b)\cdots f_{1l}(b)f_{21}(b)f_{22}(b)\cdots f_{2m}(b)g_1(b)\cdots g_n(b)\right) ^{1+ \epsilon } $$ for infinitely 
many positive integers $b$, since $ f_{11}(a)f_{12}(a)\cdots f_{1l}(a)$ and\\ $f_{21}(a)f_{22}(a)\cdots f_{2m}(a)$ have positive leading coefficients. So 
\begin{eqnarray*}
 g(b) &\leq& c\left( f_{11}(b)\cdots f_{1l}(b)f_{21}(b)\cdots f_{2m}(b)g_1(b)\cdots g_n(b)\right)\\
&&\times \left( f_{11}(a)\cdots f_{1l}(a)f_{21}(a)\cdots f_{2m}(a)g_1(b)\cdots g_n(b)\right)^{\epsilon}.
\end{eqnarray*}
Also as for $b\geq M$ for some positive integer $M$, 
{\tiny$$f_{11}(a)\cdots f_{1l}(a)f_{21}(a)\cdots f_{2m}(a)g_1(b)\cdots g_n(b)\leq b^{(\deg(f_{11}(x)\cdots f_{1l}(x)f_{21}(x)
\cdots f_{2m}(x)g_1(x)\cdots g_n(x))},$$}
we have
\begin{eqnarray*}
g(b) &\leq& c\left( f_{11}(b)\cdots f_{1l}(b)f_{21}(b)\cdots f_{2m}(b)g_1(b)\cdots g_n(b)\right)\\
&&\times b^{\epsilon \deg \left(f_{11}(x)\cdots f_{1l}(x)f_{21}(x)\cdots f_{2m}(x)g_1(x)\cdots g_n(x)\right)}\\
g(b) &\leq& c\left( f_{11}(b)\cdots f_{1l}(b)f_{21}(b)\cdots f_{2m}(b)g_1(b)\cdots g_n(b)\right)b. 
\end{eqnarray*}
Since $\epsilon \deg \left(f_{11}(x)f_{12}(x)\cdots f_{1l}(x)f_{21}(x)f_{22}(x)\cdots f_{2m}(x)g_1(x)\cdots g_n(x)\right) < 1$,
$\deg g(y)< 1+ \deg (f_{11}(x)f_{12}(x) \cdots f_{1l}(x) f_{21}(x) \cdots f_{2m}(x) g_1(x) \cdots g_n(x))$. This is a contradiction to (ii). 
So our assumption is wrong. There are only finitely many positive integral solutions $(x, y)$ for (\ref{e1}).
Second we shall prove that there are only finitely many integral solutions $(x, y) = (a, b)$ for (\ref{e1}). Let $(a, b)$ be any integral solutions for (\ref{e1}). 
Then $(a, b)$ can be written in the form $(\pm a_{1}, \pm b_{1})$, where $a_{1}, b_{1} \in \N$. Now we replace $(x, y)$  by $(\pm x_{1}, \pm y_{1})$ in (\ref{e1}). 
Then we get that
$$ f(\pm x) = g(\pm y).$$
That is,
$$f_{11}(\pm x)^{\alpha_1}\cdots f_{1l}(\pm x)^{\alpha_l}\pm f_{21}(\pm x)^{\beta_1}f_{22}(\pm x)^{\beta_2}
\cdots f_{2m}(\pm x)^{\beta_m}=g_1(\pm y) \cdots g_n(\pm y).$$
We can reduce this equation as either 
 $$f_{111}(x)^{\alpha_1}\cdots f_{11l}(x)^{\alpha_l}\pm f_{121}(x)^{\beta_1}
\cdots f_{12m}(x)^{\beta_m}=g_{11}(y) \cdots g_{1n}(y).$$
or
 $$f_{121}(x)^{\beta_1}\cdots f_{12m}(x)^{\beta_m} \pm f_{111}(x)^{\alpha_1}\cdots f_{11l}(x)^{\alpha_l}=g_{11}(y) 
\cdots g_{1n}(y).$$
It is clear that the above two equations satisfy the conditions (1), (2), (3) and (4). So the above two equations have only finitely many integral solutions. 
This proves the theorem.$\hfill\square$
\vskip.5mm
\noin{\bf Proof of Theorem \ref{abcthm2}}:
Suppose that there are infinitely many solutions $(x,y)\in\Z\times\Z$ for equation $(\ref{e1})$. First we assume that the equation $(\ref{e1})$ has 
infinitely many solutions $(a,b)$ in positive integers. Then $$f(a)=g(b)$$ for infinitely many $(a,b)\in\N\times\N$. Therefore 
\begin{equation}
f_{11}(a)^{\alpha_1}f_{12}(a)^{\alpha_2}\cdots f_{1l}(a)^{\alpha_l}\pm f_{21}(a)^{\beta_1}f_{22}(a)^{\beta_2}\cdots f_{2m}(a)^{\beta_m}=g(b)
\label{subburameq}
\end{equation}
for infinitely many $(a,b)\in\N\times\N$.
Now we shall prove that for infinitely many $(a,b)\in \N\times \N$, 
\begin{equation}
\gcd ({f_{11}(a) \cdots f_{1l}(a), f_{21}(a) \cdots f_{2m}(a)})=1.
\label{subbugcd}
\end{equation}
Suppose for infinitely many $(a,b)$, $$\gcd ({f_{11}(a) \cdots f_{1l}(a), f_{21}(a) \cdots f_{2m}(a)})>1,$$
then by division algorithm, we will have a non constant common factor of the polynomials $f_{11}(x) \cdots f_{1l}(x)$ and $f_{21}(x) \cdots f_{2m}(x)$ which 
is a contradiction to the hypothesis $(3)$ of the theorem. This establishes equation $(\ref{subbugcd})$.
Since $f_{11}(x) \cdots f_{1l}(x)$, $f_{21}(x) \cdots f_{2m}(x)$ and $g_1(y) \cdots g_n(y)$ have positive leading coefficients, we get positive integers $N_1$ 
and $N_2$ such that for all integers $x\geq N_1$ and $y\geq N_2$, $f_{11}(x) \cdots f_{1l}(x)$, $f_{21}(x) \cdots f_{2m}(x)$ and $g_1(y) \cdots g_n(y)$ are 
positive integers. Let $\epsilon>0$ be a real number such that 
$$\epsilon\deg (f_{11}(x) \cdots f_{1l}(x)f_{21}(x) \cdots f_{2m}(x)g_1(x) \cdots g_n(x))<1.$$
\noin{\bf Case i.} Assume that 
$$f_{11}(a)^{\alpha_1}f_{12}(a)^{\alpha_2}\cdots f_{1l}(a)^{\alpha_l}+ f_{21}(a)^{\beta_1}f_{22}(a)^{\beta_2}\cdots f_{2m}(a)^{\beta_m}=g(b)$$ for infinitely 
many $(a, b) \in \N \times \N$ with
\begin{enumerate}
 \item $a\geq N_{1}$ and $b \geq N_{2}$
 \item $f_{11}(a)^{\alpha_1}\cdots f_{1l}(a)^{\alpha_l}, f_{21}(a)^{\beta_1}\cdots f_{2m}(a)^{\beta_m},g(b) \in \N $
\item $f_{11}(a)^{\alpha_1}\cdots f_{1l}(a)^{\alpha_l}, f_{21}(a)^{\beta_1}\cdots f_{2m}(a)^{\beta_m},g(b)$ are coprimes.  
\end{enumerate}
So by $ABC$ Conjecture, for any fixed $\epsilon>0$, there exists a constant $c_{1}$ such that 
\begin{eqnarray*}
g(b) &\leq& c_{1}rad\left( f_{11}(a)^{\alpha_1}\cdots f_{1l}(a)^{\alpha_l}f_{21}(a)^{\beta_1}\cdots f_{2m}(a)^{\beta_m}g(b)\right) ^{1+ \epsilon }\\  
&=& c_{1}rad\left( f_{11}(a)^{\alpha_1}\cdots f_{1l}(a)^{\alpha_l}f_{21}(a)^{\beta_1}
\cdots f_{2m}(a)^{\beta_m}g_1(b)^{\gamma_1} \cdots g_n(b)^{\gamma_n}\right) ^{1+ \epsilon }.
\end{eqnarray*}
 So
$$g(b) \leq c_{1}rad\left( f_{11}(a)\cdots f_{1l}(a)f_{21}(a)\cdots f_{2m}(a)g_1(b)\cdots g_n(b)\right) ^{1+ \epsilon }.$$
Therefore 
$$g(b) \leq c_{1}\left( f_{11}(a)\cdots f_{1l}(a)f_{21}(a)\cdots f_{2m}(a)g_1(b)\cdots g_n(b)\right) ^{1+ \epsilon }.$$
\noin{\bf Case ii.} Assume that
\begin{equation}
f_{11}(a)^{\alpha_1}\cdots f_{1l}(a)^{\alpha_l} - f_{21}(a)^{\beta_1}\cdots f_{2m}(a)^{\beta_m}=g(b)
\label{equation5}
\end{equation}
for infinitely many $(a, b) \in \N \times \N$ with
\begin{enumerate}
 \item $a\geq N_{1}$ and $b \geq N_{2}$
 \item $f_{11}(a)^{\alpha_1}\cdots f_{1l}(a)^{\alpha_l}, f_{21}(a)^{\beta_1}\cdots f_{2m}(a)^{\beta_m},g(b) \in \N $
\item $f_{11}(a)^{\alpha_1}\cdots f_{1l}(a)^{\alpha_l}, f_{21}(a)^{\beta_1}\cdots f_{2m}(a)^{\beta_m},g(b)$ are coprimes.  
\end{enumerate}
So by $ABC$ Conjecture, for any fixed $\epsilon>0$, there exists a constant $c_{2}$ such that 
$$ f_{11}(a)^{\alpha_1}\cdots f_{1l}(a)^{\alpha_l}\leq c_2rad\left( f_{11}(a)^{\alpha_1}
\cdots f_{1l}(a)^{\alpha_l}f_{21}(a)^{\beta_1}\cdots f_{2m}(a)^{\beta_m}g(b)\right) ^{1+ \epsilon } $$ 
$$= c_2rad\left( f_{11}(a)^{\alpha_1}\cdots f_{1l}(a)^{\alpha_l}f_{21}(a)^{\beta_1}
\cdots f_{2m}(a)^{\beta_m}g_1(b)^{\gamma_1} \cdots g_n(b)^{\gamma_n}\right) ^{1+ \epsilon }.$$
Also as by equation (\ref{equation5}) $g(b)\leq f_{11}(a)^{\alpha_1}\cdots f_{1l}(a)^{\alpha_l}$, we have
\begin{eqnarray*}
g(b) &\leq& c_{2}rad\left( f_{11}(a)^{\alpha_1}\cdots f_{1l}(a)^{\alpha_l}f_{21}(a)^{\beta_1}
\cdots f_{2m}(a)^{\beta_m}g(b)\right) ^{1+ \epsilon }\\ 
&=& c_2rad\left( f_{11}(a)^{\alpha_1}\cdots f_{1l}(a)^{\alpha_l}f_{21}(a)^{\beta_1}
\cdots f_{2m}(a)^{\beta_m}g_1(b)^{\gamma_1} \cdots g_n(b)^{\gamma_n}\right) ^{1+ \epsilon }.
\end{eqnarray*}
So
$$g(b) \leq c_{2}rad\left( f_{11}(a)\cdots f_{1l}(a)f_{21}(a)\cdots f_{2m}(a)g_1(b)\cdots g_n(b)\right) ^{1+ \epsilon }.$$
Therefore 
$$g(b) \leq c_{2}\left( f_{11}(a)\cdots f_{1l}(a)f_{21}(a)\cdots f_{2m}(a)g_1(b)\cdots g_n(b)\right) ^{1+ \epsilon }.$$
So from Case i and Case ii, we get 
$$ g(b) \leq c\left( f_{11}(a)\cdots f_{1l}(a)f_{21}(a)\cdots f_{2m}(a)g_1(b)\cdots g_n(b)\right) ^{1+ \epsilon },$$
where $c = \max \left \lbrace c_{1}, c_{2}\right\rbrace $.
\begin{equation}
 f(a) \leq c\left( f_{11}(a)\cdots f_{1l}(a)f_{21}(a)\cdots f_{2m}(a)g_1(b)\cdots g_n(b)\right) ^{1+ \epsilon },
\label{eq5}
\end{equation}
since $$f(a) = g(b).$$
If $b > a$ for infinitely many $(a, b) \in \N \times \N $, then $f(b) > f(a)$, since leading coefficients of $f(x)$ is positive. Therefore
$$f(b) > f(a) = g(b).$$
So
$$f(b) > g(b).$$
This implies that $\deg g(x) < \deg f(x)$. This contradicts to $(i)$. So $b \leq a $ for infinitely many $(a, b) \in \N \times \N$. 
Also since $ f_{11}(x)\cdots f_{1l}(x)$ and $f_{21}(x)\cdots f_{2m}(x)$ have positive leading coefficients, by equation (\ref{eq5}), 
we get
\begin{eqnarray*} 
 f(a) &\leq& c\left( f_{11}(a)\cdots f_{1l}(a)f_{21}(a)\cdots f_{2m}(a)g_1(a)\cdots g_n(a)\right)\\
&&\times \left( f_{11}(a)\cdots f_{1l}(a)f_{21}(a)\cdots f_{2m}(a)g_1(b)\cdots g_n(b)\right)^{\epsilon}.
\end{eqnarray*}
Thus as for $a\geq M$ for some positive integer $M$, {\tiny$$f_{11}(a)\cdots f_{1l}(a)f_{21}(a)\cdots f_{2m}(a)g_1(b)\cdots g_n(b)\leq a^{\deg (f_{11}(x)\cdots f_{1l}(x)f_{21}(x)
\cdots f_{2m}(x)g_1(x)\cdots g_n(x))},$$} we have 
\begin{eqnarray*}
 f(a) &\leq& c\left( f_{11}(a)\cdots f_{1l}(a)f_{21}(a)\cdots f_{2m}(a)g_1(a)\cdots g_n(a)\right)\\
&&\times b^{\epsilon \deg \left(f_{11}(x)\cdots f_{1l}(x)f_{21}(x)\cdots f_{2m}(x)g_1(x)\cdots g_n(x)\right)}\\
f(a) &\leq& c\left( f_{11}(a)\cdots f_{1l}(a)f_{21}(a)\cdots f_{2m}(a)g_1(a)\cdots g_n(a)\right)a.
\end{eqnarray*}

\noin Since $\epsilon \deg \left(f_{11}(x)\cdots f_{1l}(x)f_{21}(x)\cdots f_{2m}(x)g_1(x)\cdots g_n(x)\right) < 1$,
$\deg f(x) < 1+ \deg (f_{11}(x) \cdots f_{1l}(x) f_{21}(x) \cdots f_{2m}(x) g_1(x) \cdots g_n(x))$. This is a contradiction to (ii). 
So there can be only finitely many positive integral solutions $(x, y)$ for (\ref{e1}).
Second we shall prove that there are only finitely many integral solutions $(x, y) = (a, b)$ for (\ref{e1}). Let $(a, b)$ be any integral solutions for (\ref{e1}). 
Then $(a, b)$ can be written as the form $(\pm a_{1}, \pm b_{1})$, where $a_{1}, b_{1} \in \N$. Now we replace (x, y)  by $(\pm x_{1}, \pm y_{1})$ in (\ref{e1}). 
Then we get that
$$ f(\pm x) = g(\pm y).$$
That is,
$$f_{11}(\pm x)^{\alpha_1}\cdots f_{1l}(\pm x)^{\alpha_l}\pm f_{21}(\pm x)^{\beta_1}\cdots f_{2m}(\pm x)^{\beta_m}=g_1(\pm y) 
\cdots g_n(\pm y)).$$
We can reduce this equation as either 
 $$f_{111}(x)^{\alpha_1}\cdots f_{11l}(x)^{\alpha_l}\pm f_{121}(x)^{\beta_1}\cdots f_{12m}(x)^{\beta_m}=g_{11}(y) \cdots g_{1n}(y)).$$
or
 $$f_{121}(x)^{\beta_1}\cdots f_{12m}(x)^{\beta_m} \pm f_{111}(x)^{\alpha_1}\cdots f_{11l}(x)^{\alpha_l}=g_{11}(y) \cdots g_{1n}(y)).$$
It is clear that the above two equations satisfy the conditions (1), (2), (3) and (4). So the above two equations have only finitely many integral solutions. 
This proves theorem.$\hfill\square$
\vskip.5mm

\noin{\bf Proof of Theorem \ref{abcthm3}}: Suppose there are infinitely many solutions $(x, y) = (a, b)$ in integers. Then $$ f(a) = g(b),$$ for infinitely many 
$(a, b) \in \Z \times \Z $. 
That is for infinitely many integers $a, b\in\Z$, we have
 $$ f(a) = g_1(b)^{\gamma_1}g_2(b)^{\gamma_2}\cdots g_n(b)^{\gamma_n}.$$ Since $\gamma_1$, ..., $\gamma_n$ are positive integers $\geq 2$. So there exists a 
positive integer $N$ such that for every integer $b > N$, $g_1(b)^{\gamma_1}g_2(b)^{\gamma_2}\cdots g_n(b)^{\gamma_n}$ is a powerful number, while 
$f(x)$ being separable of $\deg >2$, by Remark \ref{remark}, $f(x)$ can be powerful numbers only for finitely many integers $x$.  
Which is a contradiction. So our assumption is wrong.
$\hfill\square$
\vskip.5mm
\noin{\bf Acknowledgments}: The second author would like to thank the Harish Chandra Research Institute, Allahabad 
for financial support during his research study.

\end{document}